\documentclass[reqno]{amsart}

\usepackage{amsmath,amsthm,amssymb,amscd}
 \newtheorem{thm}{Theorem}[section]
 \newtheorem{cor}[thm]{Corollary}
 \newtheorem{lem}[thm]{Lemma}
 \newtheorem{prop}[thm]{Proposition}
 \theoremstyle{definition}
 
 \newtheorem{ques}[thm]{Question}
 \theoremstyle{remark}
 
 \newtheorem*{ex}{Example}
 \numberwithin{equation}{section}

\begin{document}

%-------------------------------------------------------------------------
% editorial commands: to be inserted by the editorial office
%
%\firstpage{1} \volume{228} \Copyrightyear{2004} \DOI{003-0001}
%
%
%\seriesextra{Just an add-on}
%\seriesextraline{This is the Concrete Title of this Book\br H.E. R and S.T.C. W, Eds.}
%
% for journals:
%
%\firstpage{1}
%\issuenumber{1}
%\Volumeandyear{1 (2004)}
%\Copyrightyear{2004}
%\DOI{003-xxxx-y}
%\Signet
%\commby{inhouse}
%\submitted{March 14, 2003}
%\received{March 16, 2000}
%\revised{June 1, 2000}
%\accepted{July 22, 2000}
%
%
%
%---------------------------------------------------------------------------
%Insert here the title, affiliations and abstract:

\vspace{1cm}

\title[$p$-adic quotient sets: linear recurrence sequences]
 {$p$-adic quotient sets: linear recurrence sequences}

%----------Author 1
\author[Deepa Antony]{Deepa Antony}

\address{
	Department of Mathematics \\
	Indian Institute of Technology Guwahati \\
	Assam, India, PIN- 781039}

\email{deepa172123009@iitg.ac.in}

%\thanks{This work was completed with the support of our\TeX-pert.}
%----------Author 2
\author[Rupam Barman]{Rupam Barman}
\address{Department of Mathematics \\
	Indian Institute of Technology Guwahati \\
	Assam, India, PIN- 781039}
\email{rupam@iitg.ac.in}
%----------classification, keywords, date


%\date{Revised version: November 14, 2022}
%----------additions
%\dedicatory{}
%%% ----------------------------------------------------------------------

\begin{abstract}
	Let $(x_n)_{n\geq0}$ be a linear recurrence of order $k\geq2$ satisfying
	$$x_n=a_1x_{n-1}+a_2x_{n-2}+\dots+a_kx_{n-k}$$ for all integers $n\geq k$, where $a_1,\dots,a_k,x_0,\dots, x_{k-1}\in \mathbb{Z},$ with $a_k\neq0$. 
	In [`The quotient set of $k$-generalised Fibonacci numbers is dense in $\mathbb{Q}_p$', \emph{Bull. Aust. Math. Soc.} \textbf{96} (2017), 24-29], 
	Sanna posed an open question to classify primes $p$ for which the quotient set of $(x_n)_{n\geq0}$ is dense in $\mathbb{Q}_p$. 
	In this article, we find a sufficient condition for denseness of the quotient set of the $k$th-order linear recurrence  $(x_n)_{n\geq0}$ satisfying $ x_{n}=a_1x_{n-1}+a_2x_{n-2}+\dots+a_kx_{n-k}$ 
	for all integers $n\geq k$ with initial values $x_0=\dots=x_{k-2}=0,x_{k-1}=1$, where $a_1,\dots,a_k\in \mathbb{Z}$ and $a_k=1$. We show that given a prime $p$, 
	there exist infinitely many recurrence sequences of order $k\geq 2$ so that their quotient sets are not dense in $\mathbb{Q}_p$. 
	We also study the quotient sets of linear recurrence sequences with coefficients in some arithmetic and geometric progressions.
\end{abstract}

%%% ----------------------------------------------------------------------
\maketitle
%%% ----------------------------------------------------------------------
%\tableofcontents
\section{Introduction and statement of results} 
For a set of integers $A$, the set $R(A)=\{a/b:a,b\in A, b\neq 0\}$ is called the ratio set or quotient set of $A$. Several authors have studied the denseness of ratio sets of different subsets of $\mathbb{N}$ in the positive real numbers. 
See for example \cite{real-1, real-2, real-3, real-4, real-5, real-6, real-7, real-8, real-9, real-10, real-11, real-12, real-13, real-14}. An analogous study has also been done for algebraic number fields, see for example \cite{algebraic-1, algebraic-2}. 
\par For a prime $p$, let $\mathbb{Q}_p$  denote the field of $p$-adic numbers. In recent years, the denseness of ratio sets in $\mathbb{Q}_p$ have been studied by several authors, 
see for example \cite{cubic, diagonal, Donnay, garciaetal, garcia-luca, miska, piotr, miska-sanna, Sanna1}. Let $(F_n)_{n\geq 0}$ be the sequence of Fibonacci numbers, 
defined by $F_0=0$, $F_1=1$ and $F_n=F_{n-1}+F_{n-2}$ for all integers $n\geq 2$. In \cite{garcia-luca}, Garcia and Luca showed that the ratio set of Fibonacci numbers is dense in $\mathbb{Q}_p$ for all primes $p$. 
Later, Sanna \cite[Theorem 1.2]{Sanna1} showed that, for any $k\geq 2$ and any prime $p$, the ratio set of the $k$-generalized Fibonacci numbers is dense in $\mathbb{Q}_p$. 
Sanna remarked that his result could be extended to other linear recurrences over the integers. However, he used some specific properties of the $k$-generalized Fibonacci numbers in the proof. Therefore, he made the following open question.
\begin{ques}\cite [Question 1.3]{Sanna1}\label{question-1}
	Let $(S_n)_{n\geq0}$ be a linear recurrence of order $k\geq2$ satisfying
	\begin{align*}
	S_n=a_1S_{n-1}+a_2S_{n-2}+\dots+a_kS_{n-k},
	\end{align*}
	for all integers $n\geq k$, where $a_1,\dots,a_k,S_0,\dots,S_{k-1}\in \mathbb{Z},$ with $a_k\neq0.$ For which prime numbers $p$ is the quotient set of $(S_n)_{n\geq0}$ dense in $\mathbb{Q}_p?$
\end{ques}
 In \cite{garciaetal}, Garcia et al. studied the quotient sets of certain second-order recurrences. To be specific, given two fixed integers $r$ and $s$, let $(a_n)_{n\geq 0}$ be defined by $a_{n}=ra_{n-1}+sa_{n-2}$ 
 for $n\geq 2$ with initial values $a_0=0$ and $a_1=1$; and let $(b_n)_{n\geq 0}$ be defined by $b_{n}=rb_{n-1}+sb_{n-2}$ for $n\geq 2$ with initial values $b_0=2$ and $b_1=r$.
 Garcia et al. proved the following result.
 \begin{thm}\cite[Theorem 5.2]{garciaetal}\label{thm-garcia}
 	Let $A=\{a_n: n\geq 0\}$ and $B=\{b_n: n\geq 0\}$.
 	\begin{enumerate}
 		\item[(a)] If $p\mid s$ and $p\nmid r$, then $R(A)$ is not dense in $\mathbb{Q}_p$.
 		\item[(b)] If $p\nmid s$, then $R(A)$ is dense in $\mathbb{Q}_p$.
 		\item[(c)] For all odd primes $p$,  $R(B)$ is dense in $\mathbb{Q}_p$ if and only if there exists a positive integer $n$ such that $p\mid b_n$.
 	\end{enumerate}
 \end{thm}
In this article, we study ratio sets of some other linear recurrences over the set of integers. Our results give some answers to Question \ref{question-1}. 
Our first result gives a sufficient condition for the denseness of the ratio sets of certain $k$th-order recurrence sequences. Finding a general solution to Question \ref{question-1} seems to be a difficult problem. 
Hence, in Theorem \ref{thm5}, we consider $k$th-order recurrence sequences for which $a_k=1$ and initial values $S_0=\cdots=S_{k-2}=0$, $S_{k-1}=1$. 
	Recall that a \emph{Pisot number} is a positive algebraic integer greater than $1$ all of whose conjugate elements have absolute value less than $1$.
		\begin{thm}\label{thm5}
		Let $(x_n)_{n\geq 0}$ be a $kth$-order linear recurrence satisfying
		\begin{align*} x_{n}=a_1x_{n-1}+a_2x_{n-2}+\dots+a_{k-1}x_{n-k+1}+x_{n-k}
		\end{align*}
		for all integers $n\geq k$ with initial values $x_0=x_1=\cdots=x_{k-2}=0, x_{k-1}=1$ and $a_1,\dots,a_{k-1}\in \mathbb{Z}$. 
		Suppose that the characteristic polynomial of the reccurence sequence has a root $\pm \alpha$, where $\alpha$ is a Pisot number. 
		If $p$ is a prime such that the characteristic polynomial of the recurrence sequence is irreducible in $\mathbb{Q}_p$, 
		then the quotient set of $(x_n)_{n\geq 0}$ is dense in $\mathbb{Q}_p$.
	\end{thm}
	If we take $k=3$ in Theorem \ref{thm5}, then we have the following corollary.
	\begin{cor}\label{cor5}
		Let $(x_n)_{n\geq 0}$ be a third-order linear recurrence satisfying
		\begin{align*} x_{n}=ax_{n-1}+bx_{n-2}+x_{n-3}
		\end{align*}
		for all integers $n\geq 3$ with initial values $x_0=x_1=0, x_2=1$, and the integers $a$ and $b$ are such that $(a+b)(b-a-2)<0$. If $p$ is a prime such that the characteristic polynomial of the recurrence sequence is irreducible in $\mathbb{Q}_p$,  then the quotient set of $(x_n)_{n\geq 0}$ is dense in $\mathbb{Q}_p$.
	\end{cor}

We discuss two examples as applications of Corollary \ref{cor5}.
\begin{ex} For $a\in \mathbb{N}$, let $\ell$ be an odd positive integer less than $2a$. Let $(x_n)_{n\geq 0}$ be a linear recurrence satisfying
	\begin{align*} x_{n}=ax_{n-1}+(a-\ell)x_{n-2}+x_{n-3}
	\end{align*}
	for all integers $n\geq 3$ with initial values $x_0=x_1=0, x_2=1$. 
	Then, $a$ and $b:=a-\ell$ satisfy $(a+b)(b-a-2)<0$. The characteristic polynomial is $p(x)=x^3-ax^2-(a-\ell)x-1$. Since $p(0)\neq 0$ and $p(1)=-2a+\ell\not\equiv0\pmod{2}$,  
	hence $p(x)$ is irreducible in $\mathbb{Q}_2$. Therefore, by Corollary \ref{cor5}, $R((x_n)_{n\geq 0})$ is dense in $\mathbb{Q}_2$.
\end{ex}
\begin{ex} For $a\in \mathbb{N}$ such that $3\nmid a$, let $\ell$ be an odd positive integer less than $2a$ and $3\mid \ell$. Let $(x_n)_{n\geq 0}$ be a linear recurrence satisfying
	\begin{align*} x_{n}=ax_{n-1}+(a-\ell)x_{n-2}+x_{n-3}
	\end{align*}
	for all integers $n\geq 3$ with initial values $x_0=x_1=0, x_2=1$. 
	Then, $a$ and $b=a-\ell$ satisfy $(a+b)(b-a-2)<0$. The characteristic polynomial is $p(x)=x^3-ax^2-(a-\ell)x-1$. Since $p(0)\neq 0$, $p(1)=-2a+\ell\not\equiv0\pmod{3}$ and $p(2)=-6a+2\ell+7\not\equiv0\pmod{3}$,  
	hence $p(x)$ is irreducible in $\mathbb{Q}_3$. Therefore, by Corollary \ref{cor5}, $R((x_n)_{n\geq 0})$ is dense in $\mathbb{Q}_3$.
\end{ex}
Next, we consider recurrence sequences whose $n$-th term depends on all the previous $n-1$ terms and obtain the following results.
\begin{thm}\label{thm3}
Let $(x_n)_{n\geq 0}$ be a linear recurrence satisfying
\begin{align*}
x_{n}=x_{n-1}+2x_{n-2}+\dots+(n-1)x_1+nx_0
\end{align*}
for all integers $n\geq 1$ with initial value $x_0=1$.	Then the quotient set of $(x_n)_{n\geq 0}$ is dense in $\mathbb{Q}_p$ for all primes $p$.  
\end{thm}
The recurrence relation given in Theorem \ref{thm3} generates a subsequence of the Fibonacci sequence.
\begin{thm}\label{thm4}
	Let $(x_n)_{n\geq 0}$ be a linear recurrence satisfying
	\begin{align*} 
	x_{n}=ax_{n-1}+arx_{n-2}+\dots+ar^{n-1}x_0
	\end{align*}
	for all integers $n\geq 1$, and $x_0, a,r\in \mathbb{Z}$. Then the quotient set of $(x_n)_{n\geq 0}$ is not dense in $\mathbb{Q}_p$ for all primes $p$.  
\end{thm}
In Theorem \ref{thm-garcia}, Garcia et al. studied second-order recurrence relations with specific initial values. In the following result, we consider a particular form of second-order recurrence sequence with arbitrary initial values $x_0$ and $x_1$ in the set of integers.
\begin{thm}\label{thm7}
Let $(x_n)_{n\geq 0}$ be a second-order linear recurrence satisfying $x_{n}=2ax_{n-1}-a^2x_{n-2}$ for all integers $n\geq 2$,  where $a,x_0,x_1\in\mathbb{Z}$. Then the quotient set of $(x_n)_{n\geq 0}$ is dense in $\mathbb{Q}_p$ for all primes $p$ satisfying $p\nmid a(x_1-ax_0)$. 
\end{thm}
For a prime $p$, let $\nu_p$ denote the $p$-adic valuation function. The following theorem gives a set of linear recurrence sequences of order $k$ whose ratio sets are not dense in $\mathbb{Q}_p.$
\begin{thm}\label{thm8}
		Let $(x_n)_{n\geq 0}$ be a linear recurrence of order $k\geq 2$ satisfying
	\begin{align*} 
	x_{n}=a_1x_{n-1}+\dots+a_kx_{n-k}
	\end{align*}
	for all integers $n\geq k$, where $x_0,\dots,x_{k-1}, a_1,\dots,a_k\in\mathbb{Z}$. If $p$ is a prime such that $p\nmid a_k$ and $\min\{\nu_p(a_j):1\leq j<k\}>\max\{\nu_p(x_m)-\nu_p(x_n):0\leq m,n<k\}$, then the quotient set of $(x_n)_{n\geq 0}$ is not dense in $\mathbb{Q}_p$.
\end{thm}
Next, we discuss one example as an application of Theorem \ref{thm8}. Given a prime $p$, this example gives infinitely many recurrence sequences of order $k\geq 2$ so that their quotient sets are not dense in $\mathbb{Q}_p$.
\begin{ex}
Let $(x_n)_{n\geq 0}$ be a linear recurrence of order $k\geq 2$ satisfying
\begin{align*} 
x_{n}=a_1x_{n-1}+\dots+a_kx_{n-k}
\end{align*}
for all integers $n\geq k$, where $x_0=x_1=\dots=x_{k-1}=1$ and $a_1,\dots,a_k\in\mathbb{Z}$. If $p$ is a prime such that $p|a_j,1\leq j\leq k-1$ and $p\nmid a_k$, then by Theorem \ref{thm8}, the quotient set of $(x_n)_{n\geq 0}$ is not dense in $\mathbb{Q}_p$.
\end{ex}
 \section{Preliminaries}
Let $r$ be a nonzero rational number. Given a prime number $p$, $r$ has a unique representation of the form $r= \pm p^k a/b$, where $k\in \mathbb{Z}, a, b \in \mathbb{N}$ and $\gcd(a,p)= \gcd(p,b)=\gcd(a,b)=1$. 
The $p$-adic valuation of $r$ is defined as $\nu_p(r)=k$ and its $p$-adic absolute value is defined as $\|r\|_p=p^{-k}$. By convention, $\nu_p(0)=\infty$ and $\|0\|_p=0$. The $p$-adic metric on $\mathbb{Q}$ is $d(x,y)=\|x-y\|_p$. 
The field $\mathbb{Q}_p$ of $p$-adic numbers is the completion of $\mathbb{Q}$ with respect to the $p$-adic metric. The $p$-adic absolute value can be extended to a finite normal extension $K$ over $\mathbb{Q}_p$ of degree  $n$. 
For $\alpha\in K$, define $\|\alpha\|_p$ as the $n$-th root of the determinant of the matrix of linear transformation from the vector space $K$ over $\mathbb{Q}_p$ to itself defined by $x\mapsto \alpha x$ for all $x\in K$. 
Also, $\nu_p(\alpha)$ is the unique rational number satisfying $\|\alpha\|_p=p^{-\nu_p(\alpha)}$.
%%%%%%%%%%%%%%%%%
\par We next state a few results which will be used in the proof of our theorems.
\begin{lem}\label{lem1}\cite[Lemma 2.1]{garciaetal}
	If $S$ is dense in $\mathbb{Q}_p$, then for each finite value of the $p$-adic valuation, there is an element of $S$ with that valuation.
\end{lem}
\begin{lem}\label{lem2}\cite[Lemma 2.3]{garciaetal}
	Let $A\subset\mathbb{N}$.
	\begin{enumerate}
		\item 	If $A$ is $p$-adically dense in $\mathbb{N}$, then $R(A)$ is dense in $\mathbb{Q}_p$.
		\item 	If $R(A)$ is $p$-adically dense in $\mathbb{N}$, then $R(A)$ is dense in $\mathbb{Q}_p$.
	\end{enumerate}
\end{lem}
 %We need the following result to prove Theorem \ref{thm5}.

\begin{thm}\label{thm11}\cite[Theorem 1]{brum}
	Let $\alpha_1,\dots,\alpha_n$ be units in $\Omega_p$, the completion of algebraic closure of $\mathbb{Q}_p,$ which are algebraic over the rationals $\mathbb{Q}$ and whose $p$-adic logarithms are linearly independent over $\mathbb{Q}$. 
	These logarithms are then linearly independent over the algebraic closure of $\mathbb{Q}$ in $\Omega_p$.
\end{thm}
%%%%%%%%%%%%
\section{Proof of the theorems}
We first prove Theorem \ref{thm5}.
\begin{proof}[Proof of Theorem \ref{thm5}]
		Let $p(x)=x^k-a_1x^{k-1}-a_2x^{k-2}-\dots-a_{k-1}x-1$ be the characteristic polynomial. Let $\alpha_1,\dots,\alpha_{k}$ be the $k$ distinct roots of the characteristic polynomial in its splitting field, say, $K$ over $\mathbb{Q}_p$. 
		The generating function of the sequence is  
	\begin{align*}
		t(x)=\frac{x^{k-1}}{1-a_1x-a_2x^2-\dots-x^k}=\sum_{i=1}^k\frac{1}{q(\alpha_i)}\sum_{n=0}^{\infty}\alpha_i^nx^n,
	\end{align*}
	where $q(x):=p'(x)$, the derivative of the polynomial $p(x)$. Hence, the $n$th term of the sequence is given by
	\begin{align*}
		x_n=\sum_{i=1}^{k}\frac{1}{q(\alpha_i)}\alpha_i^n,n\geq0.
	\end{align*}
	Since $p(0)=-1$, the roots  of $p(x)$ are units in the ring formed by elements in $K$ with $p$-adic absolute value less than or equal to one.
	Following Sanna's proof of \cite[Theorem 1.2]{Sanna1}, we can choose an even $t\in \mathbb{N}$ such that the function defined as
	\begin{align*}
		G(z):=\sum_{i=1}^{k}\frac{1}{q(\alpha_i)}\exp_p(z\log_p(\alpha_i^t))
	\end{align*}
	is analytic over $\mathbb{Z}_p$ and the Taylor  series of $G(z)$ around $0$ converges for all $z\in \mathbb{Z}_p$.  Also, note that $x_{nt}=G(n)$ for $n\geq0$. 
	\par
	We now use a variant of the following lemma which gives the multiplicative independence of any $k-1$ roots among the $k$ roots $\alpha_1,\dots,\alpha_k$ of  the characteristic polynomial $x^k-x^{k-1}-\dots-x-1$ of the $k$-generalized Fibonacci sequence in the field of complex numbers.
	\begin{lem}\label{lem5}\cite[Lemma 1]{fuchs}
		Each set of $k-1$ different roots $\alpha_1,\dots,\alpha_{k-1}$ is multiplicatively independent, that is, $\alpha_1^{e_1}\dots\alpha_{k-1}^{e_{k-1}}=1$ for some integers $e_1,\dots,e_{k-1}$ if and only if $e_1=\dots=e_{k-1}=0$.
	\end{lem} 
	Let  $\sigma(\alpha_1)$ be the $\pm$ Pisot root with absolute value greater than 1 and other roots, $\sigma(\alpha_2),\dots,\sigma(\alpha_k)$  with absolute values less than $1$, 
	where, $\sigma$ is an isomorphism from $\mathbb{Q}(\alpha_1,\dots,\alpha_k)$ to the splitting field of $p(x)$ over $\mathbb{Q}$ in the field of complex numbers. 
	Therefore, the proof of Lemma \ref{lem5} holds true for the roots of $p(x)$, which are $\sigma(\alpha_1),\dots, \sigma(\alpha_k)$,  since $\log|\sigma(\alpha_1)|$ is positive and $\log|\sigma(\alpha_2)|,\dots,\log|\sigma(\alpha_k)|$ are negative. 
	Hence, $\sigma(\alpha_1),\dots,\sigma(\alpha_{k-1}$) are  multiplicatively independent, implying that $\alpha_1^t,\dots, \alpha_{k-1}^t$ are multiplicatively independent. 
	Thus, $\log_p(\alpha_1^t),\dots,\log_p(\alpha_{k-1}^t)$ are linearly independent over $\mathbb{Z}$ and hence linearly independent over algebraic numbers by Theorem \ref{thm11}.
	\par
	If
	$G'(0)=\sum_{i=0}^{k}\frac{1}{q(\alpha_i)}\log_p(\alpha_i^t)=0$,
	we obtain
	\begin{align*}
		\sum_{i=1}^{k-1}	\Big(\frac{1}{q(\alpha_i)}-\frac{1}{q(\alpha_k)}\Big)\log_p(\alpha_i^t)=0
	\end{align*}
	since  $\log_p(\alpha_k^t)=-\log_p(\alpha_1^t)-\dots-\log_p(\alpha_k^t)$ as product of the roots is $-1$ and $t$ is even.
	By linear independence of $\log_p(\alpha_1^t),\dots,\log_p(\alpha_{k-1}^t)$, $\frac{1}{q(\alpha_1)}=\dots=\frac{1}{q(\alpha_k)}=c$, for some $p$-adic number $c$. This gives $k$ distinct
	roots $\alpha_1,\dots,\alpha_k$ of the $k-1$ degree  polynomial $q(x)-\frac{1}{c}$, which is not possible.
 Therefore, $G'(0)\neq0$.
	Since \[G(z)=\sum_{j=0}^\infty\frac{G^{(j)}(0)}{j!}z^j\] converges at $z=1$, so
	$\parallel\frac{G^{(j)}(0)}{j!}\parallel_p\rightarrow 0$. Hence, there exists an integer $\ell$ such that $\nu_p(G^{(j)}(0)/j!)\geq-\ell$ for all $j$.
	Thus, we obtain $G(mp^h)=G'(0)mp^h+d$ such that $\nu_p(d)\geq 2h-\ell$ for all $m,h\geq0$. Also, $G(0)=0$ for $h>h_0:=\ell+\nu_p(G'(0))$, and hence we have 
	\begin{align*}
		\nu_p\Big(	\frac{G(mp^h)}{G(p^h)}-m\Big)\geq h-h_0.
	\end{align*}
	This yields
	\[\lim_{h\rightarrow\infty}\left\lVert\frac{G(mp^h)}{G(p^h)}-m\right\rVert_p=0,\]
	and hence $R(G(n)_{n\geq 0})$ is $p$-adically dense in $\mathbb{N}$.
	Since $x_{nt}=G(n),n\geq0$, we find that  $R((x_n)_{n\geq 0})$ is also $p$-adically dense in $\mathbb{N}$.
	Therefore, by Lemma \ref{lem2},  $R((x_n)_{n\geq 0})$ is dense in $\mathbb{Q}_p$.
\end{proof}
\begin{proof}[Proof of Corollary \ref{cor5}]
	Since $p(1)p(-1)=(-a-b)(b-a-2)>0$ and $p(0)=-1$, by continuity of the polynomial function in $\mathbb{R}$, $p(x)$ has one real root with absolute value greater than 1 and other two roots  with absolute values less than $1$. 
	Hence, the characteristic polynomial has a $\pm$ Pisot root, and the corollary follows from Theorem \ref{thm5}.
\end{proof}
We will now prove Theorem \ref{thm3}. We need the following corollary to prove Theorem \ref{thm3}.
\begin{cor}\label{cor1}\cite[Corollary 2.2]{cirnu}
	The linear recurrence relation $x_{n+1}=x_n+2x_{n-1}+\dots+nx_1+(n+1)x_0,n\geq 0$ with the initial data $x_0=1$ has the solution $x_n=\frac{1}{\sqrt{5}}\left({(\frac{3+\sqrt{5}}{2})}^n-{(\frac{3-\sqrt{5}}{2})}^n\right),n\geq 1$.
\end{cor}	
\begin{proof}[Proof of Theorem \ref{thm3}]
	By Corollary \ref{cor1}, we have
	\begin{align*}
		 x_n&=\frac{1}{\sqrt{5}}\left({\left(\frac{3+\sqrt{5}}{2}\right)}^n-{\left(\frac{3-\sqrt{5}}{2}\right)}^n\right), n\geq1\\
		& =\frac{\alpha^{2n}-\beta^{2n}}{\sqrt{5}}\\
		&=F_{2n},
		\end{align*}
	   where $\alpha=\frac{1+\sqrt{5}}{2}, \beta=\frac{1-\sqrt{5}}{2}$, and $F_n$ denotes the $n$th Fibonacci number which is obtained by the Binet formula. We know by \cite{garcia-luca}, 
	   that the ratio set of Fibonacci numbers is dense in $\mathbb{Q}_p$ for all primes $p$. Therefore, by Lemma \ref{lem1}, $\nu_p(F_n)$ is not bounded. Hence, for any $j\in \mathbb{N}$, there exists $F_m$ such that $\nu_p(F_m)\geq j$, that is,
	 $\frac{\alpha^{m}-\beta^{m}}{\sqrt{5}}\equiv0 \pmod{p^j}$ which gives $\alpha^m\equiv\beta^m\pmod{p^j}$. This yields
	 \begin{align*}
	   \alpha^{2mp^{j-1}(p-1)}=(\alpha^m\alpha^m)^{p^{j-1}(p-1)}\equiv(\alpha^m\beta^m)^{p^{j-1}(p-1)}\pmod{p^j}.
	   	\end{align*}
	    Since $\alpha\beta=-1$, by using Euler's theorem, we find that 
	    \begin{align*}
	    \alpha^{2mp^{j-1}(p-1)}\equiv(\alpha^m\beta^m)^{p^{j-1}(p-1)}\equiv1\pmod{p^j}.
	    \end{align*}
	This gives $\alpha^{2k}\equiv\beta^{2k}\equiv1\pmod{p^j}$, where $k=mp^{j-1}(p-1)$. Hence,
	 \begin{align*}
	  \frac{x_{kn}}{x_{k}}=\frac{F_{2kn}}{F_{2k}}=\frac{(\alpha^{2k})^n-(\beta^{2k})^n}{\alpha^{2k}-\beta^{2k}}=(\alpha^{2k})^{(n-1)}+(\alpha^{2k})^{n-2}\beta^{2k}+\dots+(\beta^{2k})^{n-1},
	  \end{align*}
       which is congruent to $n$ modulo $p^{j}$.
	 Since for $n\in \mathbb{N}$, there exists $k\in \mathbb{N}$ such that $\parallel\frac{x_{kn}}{x_k}-n\parallel_p\leq p^{-j}$,  $R((x_n)_{n\geq 0})$ is $p$-adically dense in $\mathbb{N}$. Therefore, by Lemma \ref{lem2},  $R((x_n)_{n\geq 0})$ is dense in $\mathbb{Q}_p$.
\end{proof}
We now prove Theorem \ref{thm4}. We need the following results to prove Theorem \ref{thm4}.
\begin{thm}\label{thm10}\cite[Theorem 3.1]{cirnu}
	The numbers $x_n$ are solutions of the linear recurrence relation with constant coefficients in geometric progression $x_{n+1}=ax_n+aqx_{n-1}+\dots+aq^{n-1}x_1+aq^nx_0,n\geq 0$  with initial data $x_0$, if and only if they form the geometric progression given by the formula $x_n=ax_0(a+q)^{n-1},n\geq 1$.
	\end{thm}
\begin{lem}\label{lem3}\cite[Lemma 2.2]{garciaetal}
	If $A$ is a geometric progression in $\mathbb{Z}$, then $R(A)$ is not dense in any $\mathbb{Q}_p$.
\end{lem}
\begin{proof}[Proof of Theorem \ref{thm4}] 
	By Theorem \ref{thm10},  $(x_n)_{n\geq 1}$ forms a geometric progression where $n$th term is $ax_0(a+r)^{n-1}$ for $n\geq1$. Hence, by Lemma \ref{lem3},  $R((x_n)_{n\geq 0})$ is not dense in $\mathbb{Q}_p$ for any prime $p$.
\end{proof}
We now prove Theorem \ref{thm7}. We need few results on uniform distribution of sequence of integers to prove Theorem \ref{thm7}. Recall that a sequence  $(x_n)_{n\geq 0}$ is said to be uniformly distributed modulo $m$ if each residue occurs equally often, that is, \[\lim_{N\rightarrow\infty}\frac{\# \{n\leq N|x
	_n\equiv t \pmod{m}\}}{N}=\frac{1}{m}\] for all $t\in\mathbb{Z}$. The following theorems will be used to prove Theorem \ref{thm7}.
\begin{prop}\label{prop1}\cite[Proposition 1]{bumby}
 Suppose $\left\langle G_n\right\rangle  $ be the sequence of integers determined by the recurrence relation $G_{n+1}=AG_n-BG_{n-1}$ with initial values $G_0,G_1$ where $A,B,G_0,G_1\in\mathbb{Z}$. 	If $A=2a,B=a^2$, then $\left\langle G_n\right\rangle $ is uniformly distributed modulo a prime $p$ if and only if $p\nmid a(G_1-aG_0)$.
	\end{prop}
\begin{thm}\label{thm9}\cite[Theorem]{bumby}
Suppose $\left\langle G_n\right\rangle $ be the sequence of integers determined by the recurrence relation $G_{n+1}=AG_n-BG_{n-1}$ with initial values $G_0,G_1$ where $A,B,G_0,G_1\in\mathbb{Z}$. If $\left\langle G_n \right\rangle$ is uniformly distributed modulo $p$, then $\left\langle G_n \right\rangle$  is uniformly distributed modulo $p^h$ with $h>1$ iff: 
	\begin{enumerate}
		\item $p>3$;
		\item $p=3$ and $A^2\not\equiv B\pmod{9}$; or 
		\item $p=2, A\equiv2\pmod{4},B\equiv 1\pmod{4}$.
	\end{enumerate}
\end{thm}
\begin{proof}[Proof of Theorem \ref{thm7}] Let $p$ be a prime. The given recurrence sequence $(x_n)_{n\geq 0}$ satisfies the hypotheses of Proposition \ref{prop1}, and hence $(x_n)_{n\geq 0}$ is uniformly distributed modulo $p$. If $p>3$, then by part (1) of Theorem \ref{thm9}, $(x_n)_{n\geq 0}$ is uniformly distributed modulo $p^k$ with $k>1$, that is, \[\lim_{N\rightarrow\infty}\frac{\# \{n\leq N|x
	_n\equiv t \pmod{p^k}\}}{N}=\frac{1}{p^k}>0.\] Therefore, for all $t\in \mathbb{N}$ and  for all $k>1$, there exists $x_n$ such that $\parallel x_n-t\parallel_p\leq p^{-k}$. Hence,  $R((x_n)_{n\geq 0})$ is $p$-adically dense in $\mathbb{N}$. Therefore, by Lemma \ref{lem2},  $R((x_n)_{n\geq 0})$ is dense in $\mathbb{Q}_p$. \par 
We next consider the remaining primes $p=2,3$. From the condition $p\nmid a(x_1-ax_0)$, we have $p\nmid a$. It is easy to check that $p=3$ satisfies the condition given in part (2) of Theorem \ref{thm9} and $p=2$ satisfies the condition given in part (3) of Theorem \ref{thm9}. The rest of the proof follows similarly as shown in the case of $p>3$. This completes the proof of the theorem.
\end{proof}
We now prove Theorem \ref{thm8}. We need the following lemma to prove Theorem \ref{thm8}.
\begin{lem}\label{lem4}\cite[Lemma 3.3]{Sanna2}
	Let $(r_n)_{n\geq0}$ be a linearly recurring sequence of order $k\geq2$ given by $r_n=a_1r_{n-1}+\dots+a_k r_{n-k}$ for each integer $n\geq k$, where $r_0,\dots,r_{k-1}$ and $a_1,\dots,a_k$ are all integers. Suppose that there exists a prime number $p$ such that $p\nmid a_k$ and $\min\{\nu_p(a_j):1\leq j<k\}>\max\{\nu_p(r_m)-\nu_p(r_n):0\leq m,n<k\}$. Then $\nu_p(r_n)=\nu_p(r_{n\pmod{k}})$ for each nonnegative integer $n$. 
\end{lem}
\begin{proof}[Proof of Theorem \ref{thm8}] 
	By Lemma \ref{lem4}, we have
	\begin{align*}
		 \nu_p(x_n/x_m)=\nu_p(x_{n\pmod{k}})-\nu_p(x_{m\pmod{k}})\leq M
	\end{align*}
	 for all $n,m\in \mathbb{N}\cup\{0\}$, where $M=\max\{\nu_p(x_i):i=0,1,\dots,k-1\}$. Therefore, by Lemma \ref{lem1},  $R((x_n)_{n\geq 0})$ is not dense in $\mathbb{Q}_p$.
\end{proof}
% ------------------------------------------------------------------------
\section{Acknowledgements}
We are very grateful to the referee for the careful reading of the paper and for the comments which helped us to improve the manuscript. We thank Piotr Miska for many helpful discussions.

% ------------------------------------------------------------------------
\end{document}